\documentclass{article}
\usepackage{amssymb,amsthm,amsmath,amstext}
\usepackage{fullpage}
\usepackage{graphicx} 
\usepackage[dvipsnames]{xcolor}
\usepackage{url}
\usepackage{hyperref}
\usepackage[sortcites, doi=false, eprint=false]{biblatex}

\DeclareMathOperator{\Av}{Av}

\theoremstyle{plain}
\newtheorem{theorem}{Theorem}
\newtheorem{lemma}[theorem]{Lemma}
\newtheorem{prop}[theorem]{Proposition}

\theoremstyle{definition}
\newtheorem{defn}[theorem]{Definition}
\theoremstyle{plain}
\newtheorem{conjecture}[theorem]{Conjecture}

\title{Enumerating separable derangements}
\author{Robert Dougherty-Bliss\thanks{Dartmouth College, Department of Mathematics, Hanover, NH, USA} \and Alejandro B. Galv\'an\footnotemark[1] \and Michaela Polley\footnotemark[1] \and David Shuster\footnotemark[1]}

\date{\today}

\begin{document}
\maketitle

\begin{abstract}
    \noindent We give a polynomial-time algorithm to compute the number $b_n$ of
    separable derangements of $[n]$. This algorithm is based on a generating
    function technique which tracks permutations along with their occupied
    diagonals, where each permutation is counted once for every such diagonal.
    We provide bounds for the proportion of separable permutations which are
    derangements, show that $b_n$ and the large Schr\"oder numbers have the same
    exponential growth constant $(3 + 2 \sqrt{2})$, and use the first 3000 terms
    to conjecture more explicit asymptotic behavior. This partially answers
    several questions about separable derangements recently posed by Vatter.
\end{abstract}

\section{Introduction}

Enumerating permutations of $[n] = \{1, 2, 3, \dots, n\}$ which satisfy a given
set of restrictions is a fundamental problem in combinatorics. Vince Vatter
recently posed the question of finding the generating function for the number of
\emph{separable derangements} of $[n]$ \cite[Problem 4.4]{Vatter2026Problems}.
This work partially answers this question.

\begin{theorem}
    There exists a polynomial-time algorithm to compute $b_n$, the number of
    separable derangements of $[n]$.
\end{theorem}

At the time of writing, Vatter and Jay Pantone computed $b_n$ through $n = 18$
using a method based on ``displacement sets'' \cite[sec.~4]{Vatter2026Problems}:
\begin{align*}
    0, 1, 2, 7, 30, 124, 560, 2610, 12470, 60955, 302930, 1528621, 7790780, \\
    40202007, 208787486, 1094575377, 5766892092, 30609900691, \dots
\end{align*}
We can compute through $n = 3000$ using an entirely different approach based on
generating functions. For example,
\begin{equation*}
    b_{100}
    =
    1109494928577851681540444606156833022733492291110692095509038076266474501.
\end{equation*}
These generating functions, which we call ``diagonal generating functions,''
track occupied diagonals in permutations. This leads to, as Wilf put it,
an ``effective'' answer which can stand in ``polite society'' \cite{Wilf1982}.

Independently, the restrictions ``separable'' and ``derangement'' are
well-understood. A permutation is a derangement if it has no fixed points. The
number of derangements of $[n]$ is
\begin{equation*}
    D_n = n! \sum_{k = 0}^n \frac{(-1)^k}{k!}.
\end{equation*}
This can be obtained with inclusion-exclusion or by solving the recurrence $D_n
= (n - 1) (D_{n - 1} + D_{n - 2})$. Separable permutations are the smallest
class of permutations which contain the identity and are closed under the
operations ``sum'' and ``skew sum'':
\begin{align*}
    &(\sigma, \tau) \in S_n \times S_m \\
    &(\sigma \oplus \tau)(i) =
    \begin{cases}
        \sigma(i) & i \leq n \\
        n + \tau(i - n) & n < i \leq n + m
    \end{cases} \\
    &(\sigma \ominus \tau)(i) =
    \begin{cases}
        \sigma(i) + m & i \leq n \\
        \tau(i - n) & n < i \leq n + m.
    \end{cases}
\end{align*}
The number $s_n$ of separable permutations is given by the \emph{large
Schr\"oder numbers} (A6318 in the OEIS \cite{OEISA6318}), which for $n \geq 1$
have the explicit formula
\begin{equation*}
    s_{n + 1} = \frac{1}{n} \sum_{k = 1}^n 2^k {n \choose k} {n \choose k - 1}.
\end{equation*}
The generating function for $s_n$ is algebraic:
\begin{equation*}
    \sum_{n = 1}^\infty s_n x^n = \frac{1 - x - \sqrt{1 - 6x + x^2}}{2}.
\end{equation*}
This generating function can be obtained by applying the method of Brignall,
Huczynska, and Vatter \cite{BrignallHuczynskaVatter2008}. The method constructs
generating functions which correspond to permutation families which
contain only finitely many ``simple'' permutations, subject to the restrictions
of finitely many ``query-complete'' sets.\footnote{The terms ``simple'' and
``query-complete'' are defined in \cite{BrignallHuczynskaVatter2008}.} It is
well-known that the separable permutations are precisely $\Av(2413,
3142)$,\footnotemark and that the only simple separable permutations are 1, 12,
and 21 \cite{West1995,BoseBussLubiw1998}, so the method works.

\footnotetext{Given a set of patterns $\mathcal{S}$, the set $\Av(\mathcal{S})$
is the set of all permutations which avoid the patterns in $\mathcal{S}$.}

When we combine the restrictions ``separable'' and ``derangement,'' the situation
becomes more delicate.\footnotemark Intuitively, the difficulty is that pattern
avoidance is about relative positions while fixed points are about absolute
positions. There is no obvious relation between these two that enables one to
say when a derangement would be separable or vice versa. The general method from
\cite{BrignallHuczynskaVatter2008} will not help because ``being a derangement''
is not a finite collection of query-complete sets.

\footnotetext{The first result about derangements that avoid patterns was by
    Robertson, Saracino, and Zeilberger, who enumerated derangements in
    $\Av(\beta)$ for $\beta \in \{132, 213, 321\}$
    \cite{RobertsonSaracinoZeilberger2003}. Elizalde later extended these
    results and gave some bijective proofs with Pak
    \cite{Elizalde2004,Elizalde2011,ElizaldePak2004}.}

From a more elementary perspective, the separable structure is not that useful
when checking if something is a derangement. While the direct sum is
simple---$\sigma \oplus \tau$ is a derangement if and only if $\sigma$ and
$\tau$ are---the skew sum is not: the main diagonal of $\sigma \ominus \tau$ may
intersect $\sigma$ or $\tau$ depending on their sizes. The diagonals of
separable permutations must be ``tracked'' if we are to determine whether a
given skew sum is a derangement.

In the following sections we will outline a procedure to calculate the number of
separable derangements of $[n]$. The basic idea is to introduce a generating
function based on the bivariate sequence
\begin{equation*}
    a_{n,k} = \text{number of separable $\sigma \in S_n$ with an occupied $k$th diagonal}.
\end{equation*}
This will allow us to count the separable derangements by complementing. In
addition, Section~\ref{sec:asymptotics} contains some results about the
asymptotics of $b_n$ compared to $s_n$.

\section{Background: Decomposition of Separable Permutations}
\label{sec:background}

\begin{defn}
    A permutation is \emph{sum decomposable} (resp.~\emph{skew-sum
    decomposable}) if it can be expressed as the sum (resp.~skew sum) of two nonempty
    permutations.
\end{defn}

\begin{defn}
    Define the following objects:
    \begin{itemize}
        \item $\mathcal{T}$ is the set of separable permutations, with ordinary generating function
        \begin{equation*}
            T(x) = \sum_{n = 1}^\infty s_n x^n.
        \end{equation*}
        \item $\mathcal{T}^\oplus$ is the set of sum decomposable separable permutations, with ordinary generating function
        \begin{equation*}
            T^\oplus(x) = \sum_{n = 1}^\infty s_n^+ x^n.
        \end{equation*}
        \item $\mathcal{T}^\ominus$ is the set of skew-sum decomposable separable permutations, with ordinary generating function
        \begin{equation*}
            T^\ominus(x) = \sum_{n = 1}^\infty s_n^- x^n.
        \end{equation*}
    \end{itemize}
\end{defn}

Every sum decomposable separable permutation $\sigma$ can be uniquely
decomposed as $\sigma = \rho\oplus\tau$, subject to the condition that $\tau \in
\mathcal{T}$ and either $\rho = 1$ or $\rho \in\mathcal{T}^\ominus$. Indeed,
this decomposition is obtained by taking the $\rho$ of minimum length such that
there exists $\tau$ with $\sigma = \rho\oplus\tau$. The minimality condition
ensures that $\rho\notin\mathcal{T}^\oplus$; and if we picked any larger $\rho'$
such that $\sigma = \rho'\oplus\tau'$, then necessarily $\rho'$ is sum
decomposable as $\rho' = \rho\oplus\alpha$ for some $\alpha$.

Similarly, every skew-sum decomposable, separable permutation $\sigma$ can be uniquely decomposed as $\sigma = \rho\ominus\tau$, where $\tau\in\mathcal{T}$ and either $\rho = 1$ or $\rho\in\mathcal{T}^\oplus$. These decompositions are illustrated in Figure \ref{fig:decomposition}.
\begin{figure}
    \centering
    \includegraphics[width=0.75\linewidth]{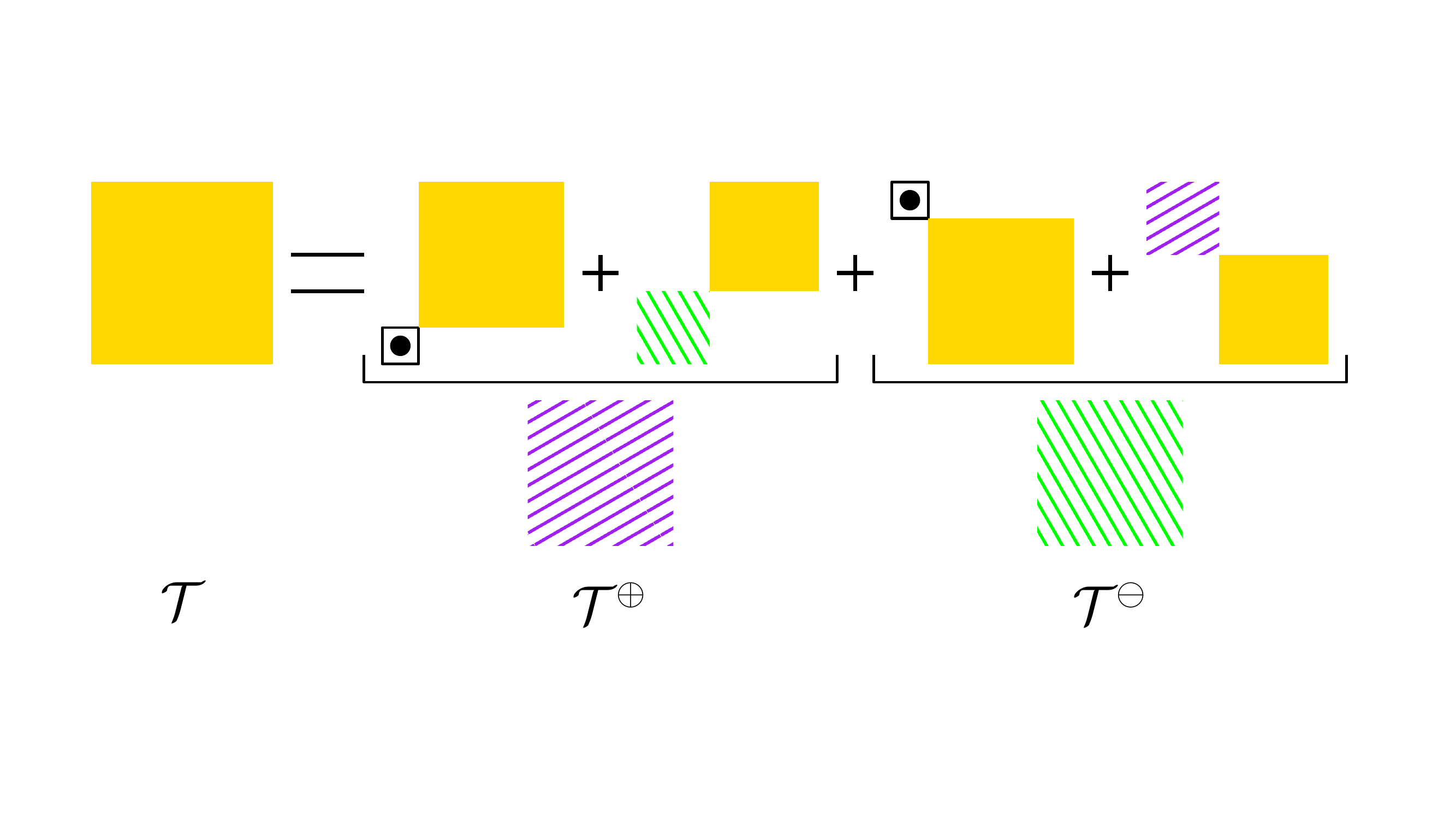}
    \caption{Unique decompositions of sum decomposable and skew-sum decomposable separable permutations.}
    \label{fig:decomposition}
\end{figure}
\begin{theorem}
    The ordinary generating functions satisfy
    \begin{equation}
        \label{background-system}
        \begin{cases}
        T(x) = x + T^\oplus(x) + T^\ominus(x),\\
        T^\oplus(x) = (x + T^\ominus(x))T(x),\\
        T^\ominus(x) = (x + T^\oplus(x))T(x).
        \end{cases}
    \end{equation}
\end{theorem}

For a proof of this theorem that uses the unique sum and skew-sum
decompositions, see \cite[Section~2]{AlbertHombergerPantone2015}. Note that
their convention includes the empty permutation; replace their $S(x)$ with $1 +
T(x)$ to obtain the same system.

\section{Diagonal Generating Functions}
\label{sec:dgfs}

The diagonals of a permutation refer to the diagonals in its graph as a
function. For a permuation of $[n]$, we number the diagonals from $-(n-1)$ to
$(n-1)$, starting from the top-left, so that the main diagonal has $k=0$ (see
Figure \ref{fig:defn_diag_gen}). In other words, $\sigma \in S_n$ has an entry
on diagonal $k$ if $i - \sigma(i) = k$ for some $i \in [n]$.

In order to keep track of occupied diagonals of a permutation, we define a new
bivariate generating function called the \emph{diagonal generating function}.
\begin{defn}
    \label{defn:dgf}
    Let $\mathcal{C}$ be a set of permutations. Then, the \emph{diagonal generating function} for $\mathcal{C}$ is given by 
    \begin{equation*}
        C(x,q)=\sum_{n, k} a_{n,k}x^nq^k
    \end{equation*}
    where $a_{n,k}$ is the number of permutations in $\mathcal{C}$ of length $n$
    that have at least one point on diagonal $k$.
\end{defn}

\begin{figure}[htbp]
    \centering
    \includegraphics[width=0.5\linewidth]{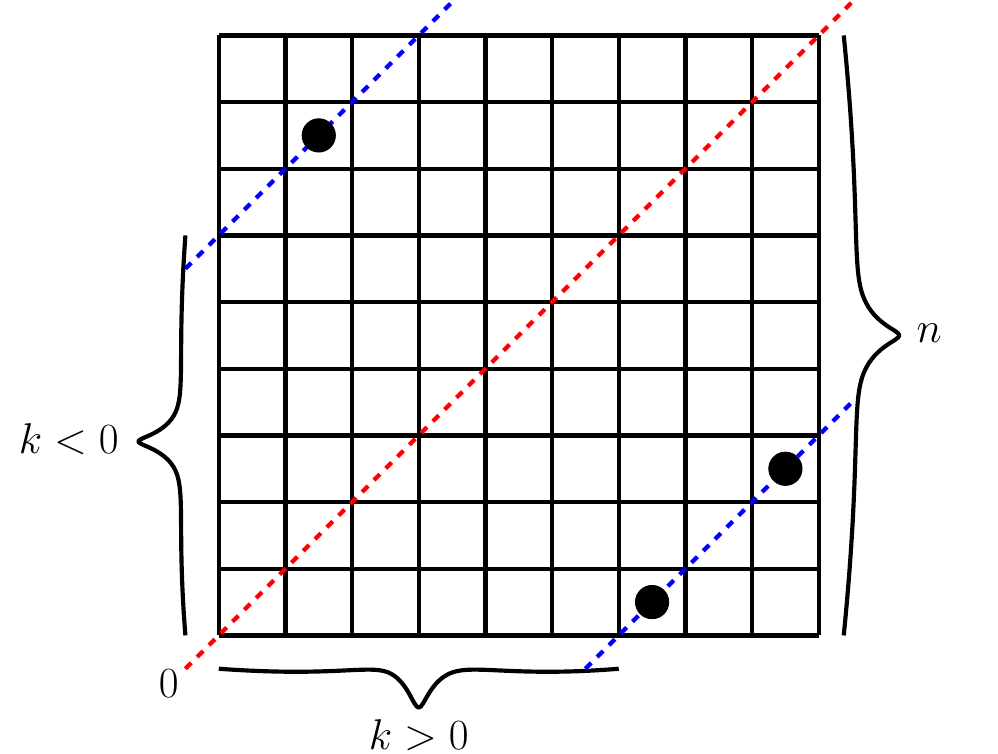}
    \caption{Diagonals as indexed by Definition~\ref{defn:dgf}. A permutation of
    $[n]$ in $\mathcal{C}$ with occupied diagonal $k$ contributes a term of $x^n
    q^k$ to the diagonal generating function of $\mathcal{C}$.}
    \label{fig:defn_diag_gen}
\end{figure}

For example, let $\displaystyle \mathcal{C}=\bigcup_{n\in \mathbb{N}} S_n$. The diagonal generating function starts \begin{equation*}
    C(x,q) =q^0x^1+\left(q^{-1}+q^0+q^1\right)x^2+\left(2q^{-2}+3q^{-1}+4q^0+3q^1+2q^2\right)x^3+\dots,
\end{equation*}
as shown in Figure \ref{fig:sn_example}.
\begin{figure}[htbp]
    \centering
    \includegraphics[width=0.95\linewidth]{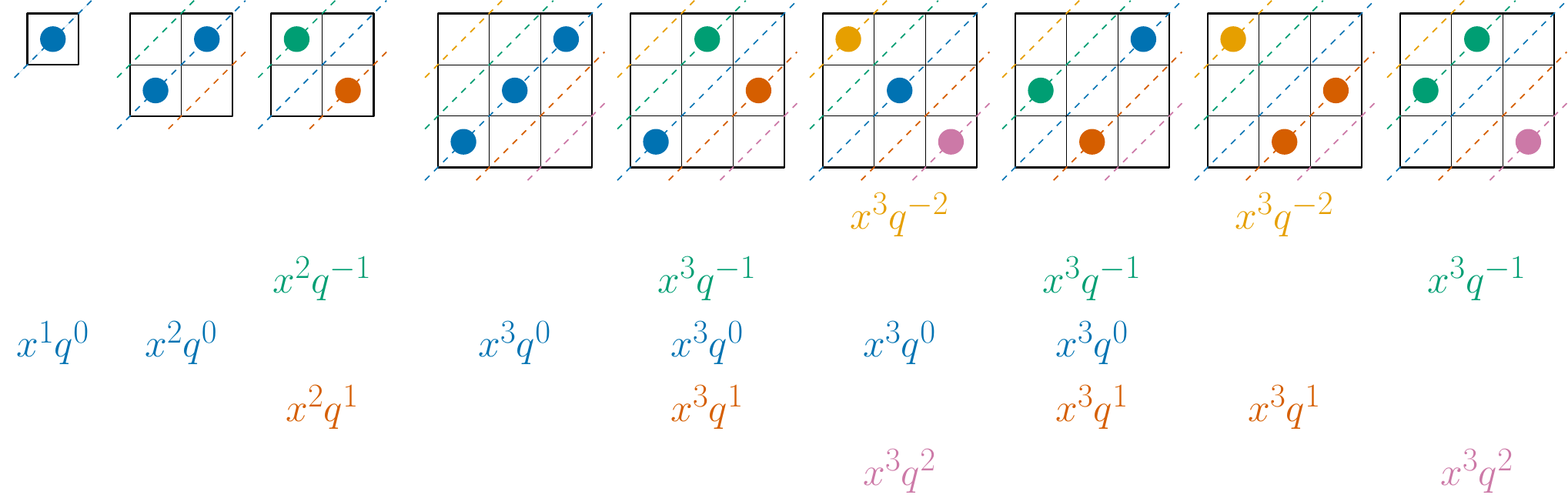}
    \caption{Permutations and the terms they contribute to diagonal generating
    functions.}
    \label{fig:sn_example}
\end{figure}

It is important to note that, in general, the diagonal generating function for $\mathcal{C}$ is not a refinement of the ordinary generating function for $\mathcal{C}$. That is, there is no value that we can substitute for $q$ which always recovers the ordinary generating function. This is because the different permutations in $S_n$ are counted a different number of times. For example, the identity permutation, $12\dots n$, is counted exactly once since all of its points lie on the main diagonal; however, the anti-identity permutation, $n\dots 21$, is counted $n$ times since all of its points lie on different diagonals. 

\section{Separable Derangements}
\label{sec:sepder-gf}

In this section, we use the structure of separable permutations to deduce a system of equations of generating functions.

\begin{defn}
    \label{defn:sequences}
    Define the following objects:
\begin{itemize}
    \item $T(x,q) = \sum_{n, k} a_{n,k} x^n q^k$ is the diagonal generating function for separable permutations.
    \item $T^\oplus(x,q) = \sum_{n, k} a^+_{n, k} x^n q^k$ is the diagonal generating function for sum decomposable separable permutations.
    \item $T^\ominus(x,q) = \sum_{n, k} a^-_{n,k} x^n q^k$ is the diagonal generating function for skew-sum decomposable separable permutations.
\end{itemize}
    All coefficients are understood to vanish outside of their natural supports.
    For example, $a_{n, k} = 0$ if $|k| \geq n$ or $n < 0$.
\end{defn}

Our objective is to express $T(x,q)$, $T^\oplus(x,q)$ and $T^\ominus(x,q)$ in
terms of each other, $T(x)$, $T^\oplus(x)$ and $T^\ominus(x)$, obtaining a
system of equations that implicitly determines $T(x,q)$. For this purpose, we
will use the same structural decomposition as in Section \ref{sec:background},
but in this instance we will need to keep track of the diagonals.

Since $\mathcal{T}^\oplus$ and $\mathcal{T}^\ominus$ are disjoint, and the only element of $\mathcal{T}$ that does not belong to either of them is the permutation of length $1$, we have
\begin{equation}\label{eq:separable}
    T(x,q) = xq^0 + T^\oplus(x,q) + T^\ominus(x,q).
\end{equation}
In order to obtain an equation for $T^\ominus(x,q)$, one needs to understand how the skew-sum operation affects the numbering of the diagonals. Let $\rho,\tau$ be two permutations of lengths $i$ and $k$, and let $\sigma = \rho\ominus\tau$. Then, $\sigma$ has length $i + k$. The $j$-th diagonal of $\rho$ corresponds to the $(j - k)$-th diagonal of $\sigma$, and the $\ell$-th diagonal of $\tau$ corresponds to the $(\ell + i)$-th diagonal of $\sigma$, as shown in Figure \ref{fig:skew_sum}.

As explained in Section \ref{sec:background}, a permutation $\sigma\in\mathcal{T}^\ominus$ can be decomposed uniquely as $\rho\ominus\tau$, where $\rho$ is either $1$ or an element of $\mathcal{T}^\oplus$, and $\tau$ is in $\mathcal{T}$. Let $i$ and $k$ be the lengths of $\rho$ and $\tau$. Then, there are two ways to choose a nonempty diagonal of $\sigma$:
\begin{itemize}
    \item Choosing a nonempty diagonal of $\rho$. Let $j$ be the index of that diagonal in $\rho$. Then $\sigma$ with that marked diagonal contributes a term $x^{i + k}q^{j - k}$ to $T^\ominus(x,q)$. The total contribution of this case when $\rho = 1$ is
    \begin{equation*}
        \sum_k s_kx^{k + 1}q^{0-k} = xq^0T\left(\frac{x}{q}\right).
    \end{equation*}
    And the total contribution of this case when $\rho\in\mathcal{T}^\oplus$ is
    \begin{align*}
        \sum_{i, k, j} a_{i,j}^+s_kx^{i + k}q^{j - k} &= \sum_{i, j} a_{i,j}^+x^iq^j\sum_k s_kx^kq^{-k} \\
        &=T^\oplus(x,q)T\left(\frac{x}{q}\right).
    \end{align*}
    \item Choosing a nonempty diagonal of $\tau$. Let $\ell$ be its index in $\tau$. Then, $\sigma$ with that marked diagonal contributes a term $x^{i + k}q^{\ell + i}$ to $T^\ominus(x,q)$. The total contribution of this case when $\rho = 1$ is
    \begin{equation*}
        \sum_{k, \ell} a_{k,\ell}x^{k + 1}q^{\ell + 1} = xqT(x, q).
    \end{equation*}
    And the total contribution of this case when $\rho\in\mathcal{T}^\oplus$ is
    \begin{align*}
    \sum_{i, k, \ell} s_i^+a_{k,\ell}x^{i + k}q^{\ell + i}
    &= \sum_i s_i^+x^iq^i\sum_{k, \ell} a_{k,\ell}x^kq^\ell \\
    &= T^\oplus(xq)T(x,q).
\end{align*}
\end{itemize}
These cases are disjoint: diagonals from $\rho$ have indices at most $i - k -
1$, and diagonals from $\tau$ have indices at least $i - k + 1$.

\begin{figure}
    \centering
    \includegraphics[width=0.85\linewidth]{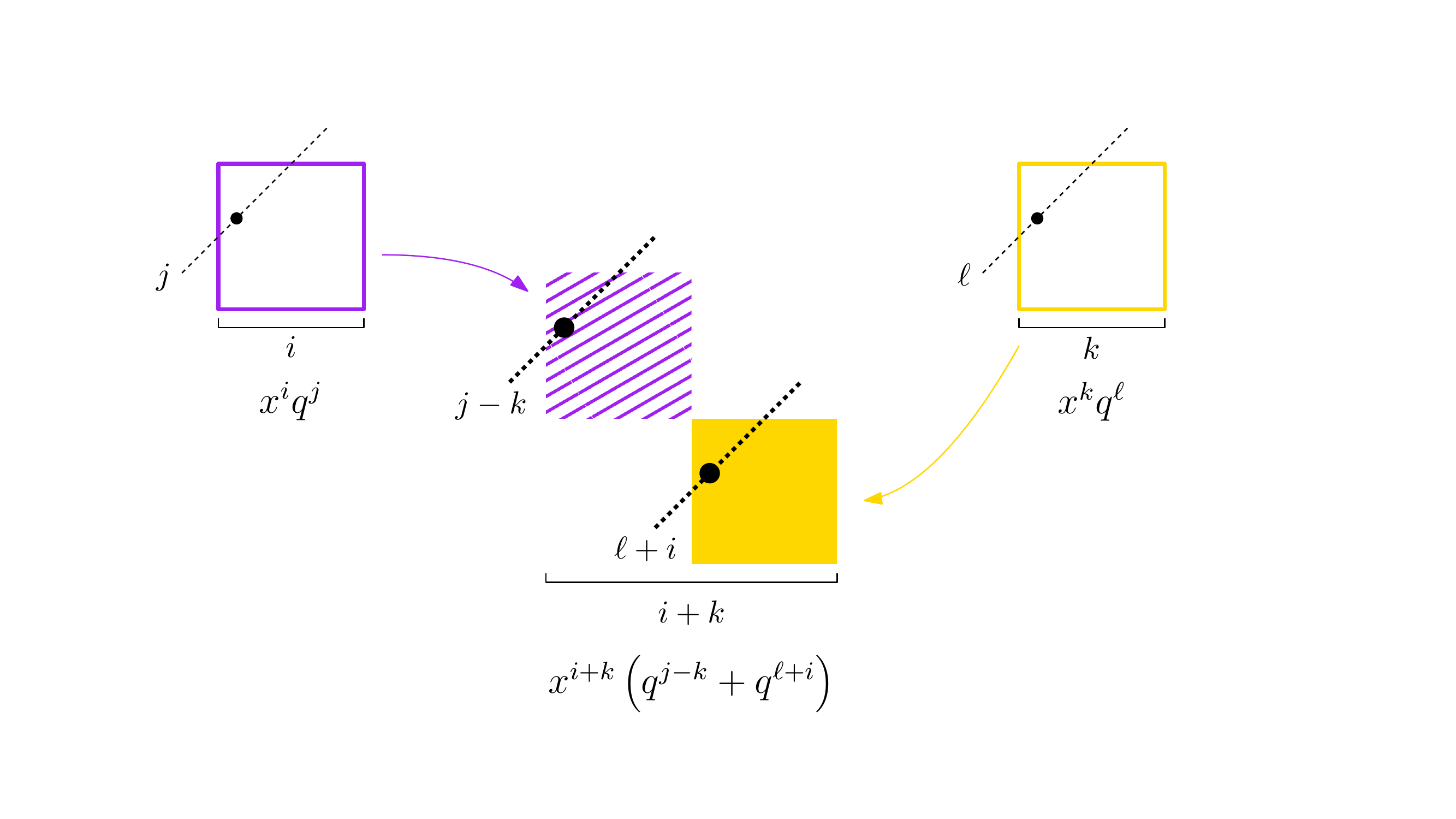}
    \caption{If $\rho$ and $\tau$ have lengths $i$ and $k$, then $\sigma = \rho\ominus\tau$ has length $i + k$. The $j$-th diagonal of $\rho$ corresponds to the $(j - k)$-th diagonal of $\sigma$, and the $\ell$-th diagonal of $\tau$ corresponds to the $(\ell + i)$-th diagonal of $\sigma$.}
    \label{fig:skew_sum}
\end{figure}

Hence,
\begin{equation}\label{eq:skew-sum}
    T^\ominus(x,q) = \left(xq^0 + T^\oplus(x,q)\right)T\left(\frac{x}{q}\right) + \left(xq + T^\oplus(xq)\right)T(x,q).
\end{equation}
In the case of $T^\oplus(x,q)$, the sum operation does not affect the numbering of the diagonals. Indeed, let $\rho,\tau$ be two permutations of lengths $i$ and $k$, and let $\sigma = \rho\oplus\tau$. Then, $\sigma$ has length $i + k$, the $j$-th diagonal of $\rho$ corresponds to the $j$-th diagonal of $\sigma$, and the $\ell$-th diagonal of $\tau$ corresponds to the $\ell$-th diagonal of $\sigma$, as shown in Figure \ref{fig:sum}.

As explained in Section \ref{sec:background}, a permutation $\sigma\in\mathcal{T}^\oplus$ can be decomposed uniquely as $\rho\oplus\tau$, where $\rho$ is either $1$ or an element of $\mathcal{T}^\ominus$, and $\tau$ is in $\mathcal{T}$. Let $i$ and $k$ be the lengths of $\rho$ and $\tau$. Then, there are two ways to choose a nonempty diagonal of $\sigma$:
\begin{itemize}
    \item Choosing a nonempty diagonal of $\rho$, with index $j$. This contributes a term $x^{i + k}q^j$ to $T^\oplus(x,q)$. The total contribution of this case is
    \begin{equation*}
        \sum_k s_kx^{k + 1}q^0 + \sum_{i, k, j} a_{i,j}^-s_kx^{i + k}q^j = \left(xq^0 + T^\ominus(x,q)\right)T(x).
    \end{equation*}
    \item Choosing a nonempty diagonal of $\tau$, with index $\ell$. This contributes a term $x^{i + k}q^\ell$ to $T^\oplus(x,q)$. The total contribution of this case is
    \begin{equation*}
        \sum_{k, \ell} a_{k,\ell}x^{k + 1}q^\ell + \sum_{i, k, \ell} s_i^-a_{k,\ell}x^{i + k}q^\ell = \left(x + T^\ominus(x)\right)T(x,q).
    \end{equation*}
\end{itemize}
However, these two cases are not disjoint. To account for the double counting, one needs to subtract the cases corresponding to diagonals that are nonempty in both $\rho$ and $\tau$. This correction term is given by
\begin{equation*}
    \sum_k a_{k,0}x^{k + 1}q^0 + \sum_{i, k, j} a_{i,j}^-a_{k,j}x^{i + k}q^j = \left(xq^0 + T^\ominus(x,q)\right)\bigodot_q T(x,q),
\end{equation*}
where $\bigodot_q$ denotes the Hadamard product (termwise multiplication) with respect to $q$. By that, we mean
\begin{equation*}
    \left( \sum_j f_j(x) q^j \right) \bigodot_q
    \left( \sum_j g_j(x) q^j \right) = \sum_j f_j(x) g_j(x) q^j.
\end{equation*}

\begin{figure}[!htb]
    \centering
    \includegraphics[width=0.85\linewidth]{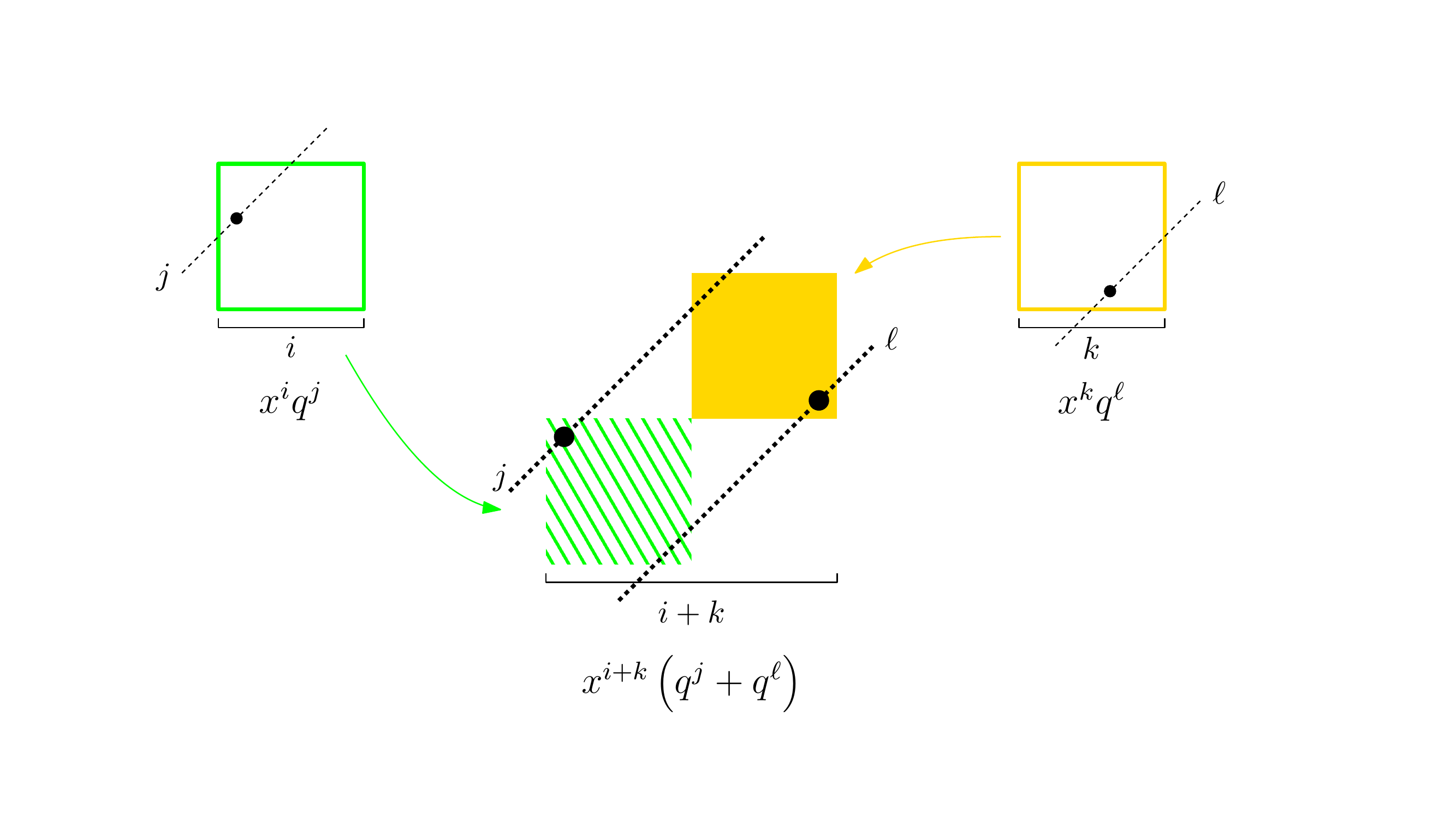}
    \caption{If $\rho$ and $\tau$ have lengths $i$ and $k$, then $\sigma = \rho\oplus\tau$ has length $i + k$. The $j$-th diagonal of $\rho$ and $\ell$-th diagonal of $\tau$ correspond to the $j$-th and $\ell$-th diagonals of $\sigma$.}
    \label{fig:sum}
\end{figure}

Hence,
\begin{equation}\label{eq:sum}
    T^\oplus(x,q) = \left(xq^0 + T^\ominus(x,q)\right)T(x) + \left(x + T^\ominus(x)\right)T(x,q) - \left(xq^0 + T^\ominus(x,q)\right)\bigodot_q T(x,q).
\end{equation}

\begin{theorem}
    The functions $T(x,q)$, $T^\oplus(x,q)$ and $T^\ominus(x,q)$ satisfy the
    following system:
    \begin{equation}
        \label{system}
        \begin{cases}
        T(x,q) &= xq^0 + T^\oplus(x,q) + T^\ominus(x,q),\\
        T^\ominus(x,q) &= \left(xq^0 + T^\oplus(x,q)\right)T\left(\frac{x}{q}\right) + \left(xq + T^\oplus(xq)\right)T(x,q),\\
        T^\oplus(x,q) &= \left(xq^0 + T^\ominus(x,q)\right)T(x) + \left(x + T^\ominus(x)\right)T(x,q) - \left(xq^0 + T^\ominus(x,q)\right)\bigodot_q T(x,q).
        \end{cases}
    \end{equation}
\end{theorem}

\section{Recurrences and implementation}
\label{sec:recurrences}

In this section we will give explicit recurrences implied by \eqref{system},
then explain some implementation details about how terms were computed. Refer to
Definition~\ref{defn:sequences} for the meaning of each sequence. The following
formulas use \emph{Iverson bracket} notation from \cite{concrete}, defined as
\begin{equation*}
    [P]
    =
    \begin{cases}
        1 & \text{if $P$ is true} \\
        0 & \text{if $P$ is false}
    \end{cases}
\end{equation*}
for any statement $P$.

Recall that our goal is to compute $s_n - a_{n, 0}$, the number of separable
derangements. Extracting coefficients on $x^n$ in \eqref{background-system} and
$x^n q^k$ in \eqref{system}, followed by some rearranging, gives a set of
interlocking recurrences. We will give them in two parts.

\begin{prop}
    The sequences $s_n$, $s^+_n$, and $s^-_n$ satisfy the recurrences
    \begin{align*}
        s_n &= s^+_n + s^-_n \\
        s^+_n &= \sum_{i = 1}^{n - 1} ([i = 1] + s^-_i) s_{n - i} \\
        s^-_n &= \sum_{i = 1}^{n - 1} ([i = 1] + s^+_i) s_{n - i}
    \end{align*}
    for $n \geq 2$ with initial conditions $(s_1, s^+_1, s^-_1) = (1, 0, 0)$.
\end{prop}

\begin{prop}
    The sequences $a_{n, k}$, $a^+_{n, k}$, and $a^-_{n, k}$ satisfy the
    recurrences
    \begin{align}
    \label{recurrences}
    \begin{split}
        a_{n, k} &= a^+_{n, k} + a^-_{n, k} \\
        a^+_{n,k} &= \sum_{i = 1}^{n - 1}
            \left(
                ([i = 1][k = 0] + a^-_{i, k}) (s_{n - i} - a_{n - i, k})
                +
                ([i = 1] + s^-_i) a_{n - i, k}
            \right) \\
        a^-_{n,k} &= \sum_{i = 1}^{n - 1}
            \left(
                ([i = 1][k + n - i = 0] + a^+_{i, k + n - i}) s_{n - i}
                +
                ([i = 1] + s^+_i) a_{n - i, k - i}
            \right) \\
    \end{split}
    \end{align}
    for $n \geq 2$ with initial conditions $(a_{1, k}, a^+_{1, k}, a^-_{1, k}) = ([k = 0], 0, 0)$.
\end{prop}

\subsection{Implementation details and timing}

It is straightforward to implement these recurrences on a computer. We first
implemented them in Maple by hand, then translated them to C with GPT-5.6 Sol.
The FLINT library was used to handle large-integer arithmetic \cite{flint}. The
code can be found at the repository \url{https://github.com/rdbliss/sepder}.

Below is a rough upper-bound estimate on the runtime of computing with the
recurrences \eqref{recurrences}. It accounts for the cost of large-integer
arithmetic.

\begin{theorem}
    The numbers $a_{i, k}$ with $1 \leq i \leq n$ and $|k| < i$ can be computed
    with $O(n^3 M(n))$ bit operations, where $M(n)$ is the number of bit
    operations needed to multiply integers with $n$ bits.
\end{theorem}

\begin{proof}
    The recurrences \eqref{recurrences} compute $a_{n, k}$ by ``row.'' Once the
    previous rows are known, \eqref{recurrences} requires $O(n)$ additions and
    multiplications. There are $O(n)$ entries, so the $n$th row takes $O(n^2
    M(\alpha_n))$, where $\alpha_n$ is an upper bound for the
    bit-lengths of the $a_{i, k}$ for $i < n$.

    The asymptotics of the Schr\"oder numbers $s_n$ are well-known
    \cite{OEISA6318} and imply that the bit-length of $s_n$ is $O(n)$. Since $0
    \leq a_{n, k}, a^\pm_{n, k} \leq s_n$, the same is true for $a_{n, k}$ and
    $a^\pm_{n, k}$, so we can take $\alpha_n = O(n)$. Therefore the $n$th row
    costs $O(n^2 M(n))$. Summing over all rows, and assuming that $M(n)$ is
    increasing, gives the total cost
    \begin{equation*}
        \sum_{i = 1}^n O(i^2 M(i)) \leq M(n) \sum_{i = 1}^n O(i^2) = O(n^3 M(n)). \qedhere
    \end{equation*}
\end{proof}

There are a number of performance improvements to be made for the recurrences.
First, by applying the inversion bijection we see that $a_{n, k} = a_{n, -k}$ for all $k$, so it is unnecessary to compute
$a_{n, k}$ for $k < 0$. The same is true for $a^+_{n, k}$ and $a^-_{n, k}$.
Second, the computation of $a_{n, k}$ depends only on values of $a_{i, k}$ with
$i < n$; in other words, every entry in a new row can be computed independently
in parallel. Finally, if the goal is to only compute $a_{n, k}$ for some large
$n$, then it is unnecessary to compute all $a_{i, j}$ with $1 \leq i \leq n$
and $0 \leq j < i$. The recurrence \eqref{recurrences} shows that $a_{n, k}$
only depends immediately on three lines of terms:
\begin{enumerate}
    \item $a_{i, k}$ for $i < n$
    \item $a_{i, k + n - i}$ for $(k + n) / 2 < i < n$
    \item $a_{n - i, k - i}$ for $i < (k + n) / 2$
\end{enumerate}
Repeating this idea on the points on these lines shows that only a ``diamond''
of terms need to be computed to determine $a_{n, 0}$. (Or a half-diamond if the
symmetry optimization $a_{n, k} = a_{n, -k}$ is used.) These performance
optimizations give constant-factor improvements. The asymptotic runtime is
unchanged from $O(n^3 M(n))$. Experimental runtime data from our C
implementation is in Table~\ref{tab:sepder-benchmark}.

There is one more optimization, though it is slightly more involved. Rather than
tracking \emph{occupied} diagonals, which leads to \eqref{recurrences}, it is
computationally simpler to track \emph{empty} diagonals. Diagonal $k$ is empty
in $\sigma \oplus \tau$ if and only if it is empty in $\sigma$ and $\tau$. This
eliminates the need for inclusion-exclusion in \eqref{recurrences}, which
reduces the number of multiplications in the equation for $a^+_{n,k}$. A similar
optimization works for $a^-_{n,k}$. This conceptual change does not change the
asymptotic runtime, nor does it lead to any significantly different generating
function equations.

\begin{table}[t]
    \centering
    \begin{tabular}{r|r|r}
        $n$
        & Median time (s)
        & Time range (s) \\
        \hline
         100 &  0.01 &  0.01--0.04 \\
         200 &  0.03 &  0.03--0.03 \\
         300 &  0.08 &  0.07--0.08 \\
         400 &  0.25 &  0.22--0.25 \\
         500 &  0.70 &  0.69--0.72 \\
         600 &  1.53 &  1.51--1.70 \\
         700 &  2.85 &  2.85--2.89 \\
         800 &  4.91 &  4.85--4.95 \\
         900 &  7.92 &  7.92--7.99 \\
        1000 & 12.33 & 12.21--12.89
    \end{tabular}
    \caption{Time to compute $a_{n, 0}$ for different $n$. Each row reports the
    median and range of three fresh runs. The benchmark was performed on an AMD Ryzen AI 7 445, a 6-core CPU with 12 threads.}
    \label{tab:sepder-benchmark}
\end{table}

\section{Asymptotics}
\label{sec:asymptotics}

In \cite[Question 4.5]{Vatter2026Problems}, Vatter asks the following question:
What is the limiting proportion of derangements among separable permutations,
assuming this limit exists? In this section we will provide some partial
answers.

Vatter originally suggested that $b_n / s_n$ looks to be monotonically
decreasing, and perhaps converging to a constant in $[1/5, 1/4]$. We observe
that $b_n / s_n$ is monotonically decreasing for $5 \leq n \leq 3000$, but it
seems to decrease to 0. See Figure~\ref{fig:empirical-ratio}.

\begin{figure}[t]
    \centering
    \includegraphics[width=\textwidth]{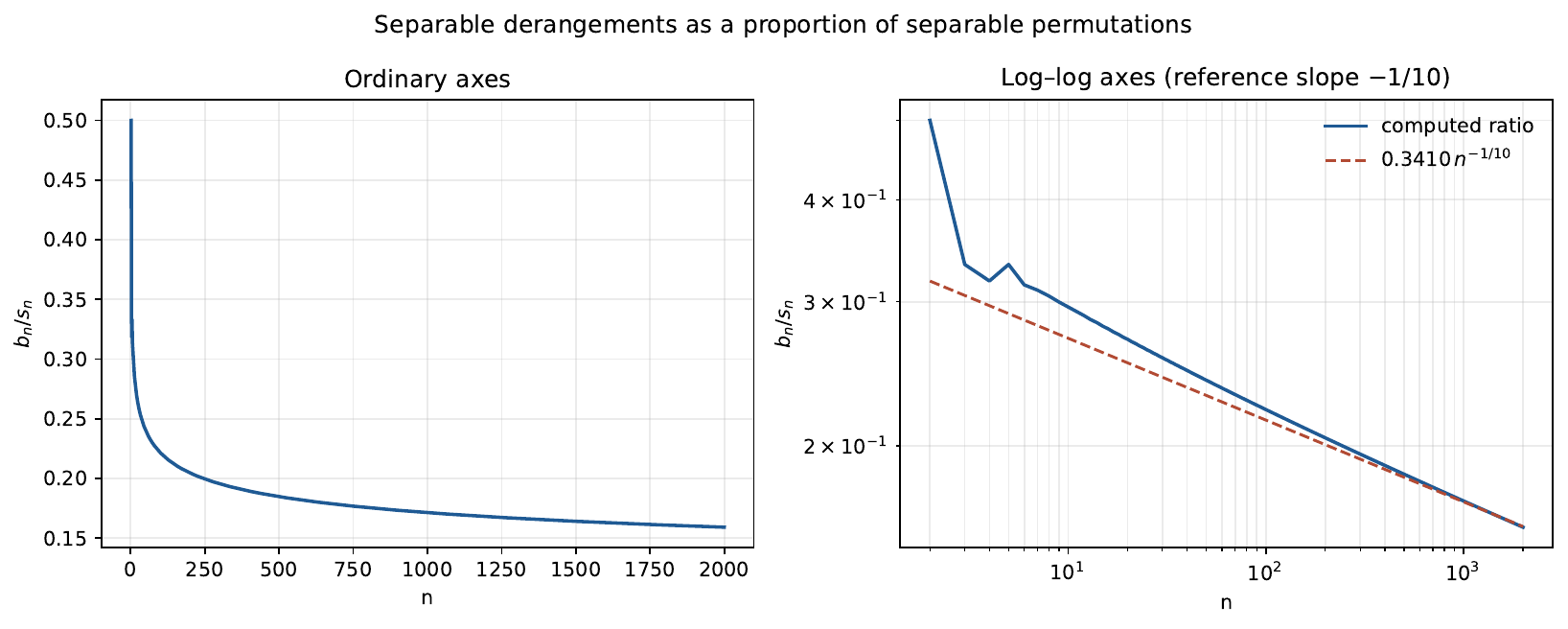}
    \caption{The ratio $b_n / s_n$ up to $n = 2000$. On the right, the dotted
    line shows the projected asymptotics if Conjecture~\ref{conjecture-asym}
    were true with $\alpha = 1.6$. The intercept was fit with numerical data.}
    \label{fig:empirical-ratio}
\end{figure}

The following results put some theoretical limits on how rare derangements can
be among the separable permutations.

\begin{lemma}
    \label{lemma:bound}
    For any integer $n \geq 1$ we have $b_{2n} \geq s_n^2$ and $b_{2n + 1} \geq
    s_n (s_{n + 1} - s_n)$. For $n \geq 3$ we have $b_n \leq s_n - 2 s_{n - 1} +
    s_{n - 2}$.
\end{lemma}

\begin{proof}
    If $\sigma, \tau \in S_n$ are any permutations of length $n$, then $\sigma
    \ominus \tau$ is a derangement. In particular, this is also true if
    $\sigma$ and $\tau$ are separable, so there are at least $s_n^2$ separable
    derangements of length $2n$.

    Let $\sigma \in S_n$ and $\tau \in S_{n + 1}$ be separable. Then $\sigma
    \ominus \tau$ is a derangement if and only if $\tau(1) \leq n$, meaning
    that $\tau$ does not begin with its maximum. Any separable permutation
    which begins with its maximal value is $1 \ominus \tau'$ for some separable
    $\tau'$ of length one less, and so there are exactly $s_n$ different $\tau$
    which would not create a derangement. Therefore, there are at least $s_n
    (s_{n + 1} - s_n)$ derangements of length $2n + 1$.

    For the upper bound, note that no separable derangement of $[n]$ starts
    with $1$ or ends with $n$. There are $s_n$ separable permutations; $s_{n -
    1}$ of them start with $1$, $s_{n - 1}$ of them end in $n$, and $s_{n - 2}$
    of them do both. By inclusion-exclusion, $b_n \leq s_n - 2 s_{n
    - 1} + s_{n - 2}$.
\end{proof}

For our later results, it is convenient to recall the asymptotics of the
Schr\"oder numbers \cite{OEISA6318}:
\begin{equation}
    \label{eq:schroder}
    s_n \sim C \lambda^n n^{-3/2}
\end{equation}
where $C$ is a positive constant and
\begin{equation*}
    \lambda = 3 + 2 \sqrt{2}.
\end{equation*}

\begin{theorem}
    There exists a constant $\beta > 0$ such that
    \begin{equation*}
        \beta n^{-3/2} \leq \frac{b_n}{s_n} \leq 0.69
    \end{equation*}
    for sufficiently large $n$.
\end{theorem}

\begin{proof}
    By Lemma~\ref{lemma:bound}, we have
    \begin{equation*}
        \frac{b_{2n}}{s_{2n}} \geq \frac{s_n^2}{s_{2n}} \sim C \left(\frac{2}{n}\right)^{3/2}.
    \end{equation*}
    This implies $b_{2n} / s_{2n} = \Omega(n^{-3/2})$. Similarly,
    \begin{equation*}
        \frac{b_{2n + 1}}{s_{2n + 1}} \geq \frac{s_n(s_{n + 1} - s_n)}{s_{2n + 1}}
        \sim C \left( \frac{2}{n} \right)^{3/2} (1 - \lambda^{-1}).
    \end{equation*}
    This implies $b_{2n + 1} / s_{2n + 1} = \Omega(n^{-3/2})$. Taken together,
    these statements imply $b_n / s_n = \Omega(n^{-3/2})$.

    For the other direction, the upper bound part of Lemma~\ref{lemma:bound}
    implies
    \begin{equation*}
        \frac{b_n}{s_n} \leq 1 - 2\frac{s_{n - 1}}{s_n} + \frac{s_{n - 2}}{s_n}.
    \end{equation*}
    Letting $n \to \infty$ in the right-hand side produces
    \begin{equation*}
        1 - \frac{2}{\lambda} + \frac{1}{\lambda^2} = 0.68629\dots.
    \end{equation*}
    In particular, $b_n / s_n \leq 0.69$ for sufficiently large $n$.
\end{proof}

\begin{theorem}
    Let $\lambda = 3 + 2 \sqrt{2}$. Then
    \begin{equation*}
        \lim_{n \to \infty} b_n^{1/n} = \lim_{n \to \infty} s_n^{1/n} = \lambda.
    \end{equation*}
\end{theorem}

\begin{proof}
    The Schr\"oder limit is a consequence of the asymptotics
    \eqref{eq:schroder} of $s_n$. For $b_n$, note that we have the inequality
    \begin{equation*}
        s_n^2 \leq b_{2n} \leq s_{2n}.
    \end{equation*}
    Taking $2n$th roots on both sides shows that $\lim_{n \to \infty}
    b_{2n}^{1/(2n)} = \lambda$. The odd $n$ case is handled similarly.
\end{proof}

The previous result shows that $b_n > (\lambda - \epsilon)^n$ is eventually true
for any $\epsilon \in (0, \lambda)$. In particular, if $b_n / s_n$ goes to $0$,
it does not do so for ``trivial'' exponential growth reasons. We conjecture the
following.

\begin{conjecture}
    \label{conjecture-asym}
    There exist constants $C' > 0$ and $\alpha \in (3/2, 3]$ such that
    \begin{equation*}
        b_n \sim C' \frac{\lambda^n}{n^\alpha}.
    \end{equation*}
    Accordingly,
    \begin{equation*}
        \frac{b_n}{s_n} = \Theta(n^{3/2 - \alpha})
    \end{equation*}
\end{conjecture}

How would we estimate $\alpha$ if Conjecture~\ref{conjecture-asym} were true? To
remove $C'$ from the picture, form quotients of terms a fixed-proportion away:
\begin{equation*}
    \frac{b_{2n}}{b_n} \sim \lambda^n 2^{-\alpha}
\end{equation*}
This relation is equivalent to
\begin{equation*}
    \log \frac{b_{2n}}{b_n} = n \log \lambda - \alpha \log 2 + o(1),
\end{equation*}
so we can use
\begin{equation}
    \label{eq:estimator}
    \alpha \approx n \log_2 \lambda - \log_2 \frac{b_{2n}}{b_n}
\end{equation}
as an estimator. Empirically it seems that the expression on the right-hand side
of \eqref{eq:estimator} converges to a constant. With $n = 1500$ we estimate
$\alpha \approx 1.605665$, and with $n = 1000$ we estimate $\alpha \approx
1.606874$. It could be that $\alpha = 1.6$, which would give $b_n / s_n =
\Theta(n^{-1/10})$. Of course, the conjecture could also be false. It would be
difficult to empirically distinguish $b_n / s_n = \Theta(n^{-1/10})$ from $b_n /
s_n = \Theta((\log n)^{-1})$ without much larger $n$.

\section{Conclusion and Open Questions}
\label{sec:conclusion}

Using diagonal generating functions, we have given an explicit polynomial-time
algorithm to compute $b_n$, the number of separable derangements of $[n]$. This
allows us to compute significantly more terms than were previously known. We
believe that diagonal generating functions may be useful in other contexts. In
particular, separable derangements are merely derangements in $\Av(2413, 3142)$;
the same ideas may apply to derangements which avoid other patterns. We intend
to investigate this in a later publication.

There are a number of open questions related to separable derangements. First,
while we have a polynomial-time algorithm to compute $b_n$, we have been unable
to conjecture an explicit formula, generating function, D-finite (P-recursive)
recurrence, or differential equation \cite{Kauers2023}. There may be a more
explicit representation of $b_n$, but we have not found one.

Second, we believe that separable derangements should, in the limit, be very
sparse among all separable permutations. More precisely, we make the following
conjecture.

\begin{conjecture}
    The ratio $b_n / s_n$ tends to $0$ as $n \to \infty$.
\end{conjecture}

\section{Acknowledgments}
We thank Vincent Vatter and Jay Pantone for bringing this problem to our
attention, and especially Jay for encouraging us to write up our results.

Artificial intelligence was used in the preparation of this manuscript. GPT-5.6
Sol was used to implement the recurrences in C, copyedit the text, and generate
some mathematical ideas. In particular, Lemma~\ref{lemma:bound} was suggested by
GPT-5.6 Sol.

\printbibliography

\end{document}